\documentclass{gtpart}
\usepackage{bbm}
\usepackage{amsopn}
\usepackage{color}
\usepackage{tikz-cd}
\usepackage{tipa}

\usepackage{geometry}
\usepackage{marginnote}
\usepackage{todonotes}

\usepackage{tikz}
\usetikzlibrary{patterns} 
\usepackage{wrapfig}
\usetikzlibrary{matrix,arrows,calc}

\usepackage{pinlabel}  
\usepackage{graphicx}  
\usepackage[all]{xy}
\usepackage{amscd}

\newtheorem{theorem}{Theorem}[section]
\newtheorem{proposition}[theorem]{Proposition}
\newtheorem{lemma}[theorem]{Lemma}
\newtheorem{corollary}[theorem]{Corollary}

\theoremstyle{definition}
\newtheorem{definition}[theorem]{Definition}
\newtheorem{bigremark}[theorem]{Remark}

\newtheorem{example}[theorem]{Example}
\newtheorem{construction}[theorem]{Construction}

\newcommand\scalemath[2]{\scalebox{#1}{\mbox{\ensuremath{\displaystyle #2}}}}

\newcounter{bean}

\newcommand{\mc}[1]{\ensuremath{\mathcal{#1}}}

\renewcommand{\Z}{\ensuremath{\mathbb{Z}}}
\renewcommand{\Q}{\ensuremath{\mathbb{Q}}}

\newcommand{\zk}{\mathcal{Z}_K}

\newcommand{\K}{K}

\newcommand{\bk}{\mathbf{k}}

\newcommand{\stab}{\mathrm{stab}}
\newcommand{\Aut}{\mathrm{Aut}}

\title{Simplicial $G$-complexes and representation stability of polyhedral products}

\author{Xin Fu}  \givenname{Xin} \surname{Fu}
\address{School of Mathematics, University of Southampton,Southampton SO17 1BJ, United Kingdom} 
\email{X.Fu@soton.ac.uk} 
\author{Jelena Grbi\'c}  \givenname{Xin} \surname{Fu}
\address{School of Mathematics, University of Southampton,Southampton SO17 1BJ, United Kingdom} 
\email{J.Grbic@soton.ac.uk} 
 \subject{primary}{msc2010}{20C30} \keyword{} \subject{primary}{msc2010}{05E10} \subject{secondary}{msc2010}{55N91} \subject{secondary}{msc2010}{55U10}

\begin{document}

\begin{abstract}
Representation stability in the sense of Church-Farb is concerned with stable properties of representations of sequences of algebraic structures, in particular of groups. We study this notion on objects arising in toric topology.
With a simplicial $G$-complex $K$ and a topological pair $(X, A)$, a $G$-polyhedral product $(X, A)^K$ is associated. We show that the homotopy decomposition~\cite{BBCG} of $\Sigma (X, A)^K$ is then $G$-equivariant after suspension. 
In the case of $\Sigma_m$-polyhedral products, we give criteria on simplicial $\Sigma_m$-complexes which imply representation stability of $\Sigma_m$-representations $\{H_i((X, A)^{K_m})\}$.
\end{abstract}
\maketitle

\section{Introduction}

Church and Farb~\cite{CF} introduced the theory of representation stability. The goal of representation stability is to provide a framework for generalising the classical homology stability to situations when each vector space $V_m$ has an action of the symmetric group $\Sigma_m$ (or other natural families of groups). 
We initiate the study of representation stability to toric topology.

Let $K$ be a simplicial complex on $m$ vertices. With $K$, and a topological pair $(X,A)$, a polyhedral product $(X,A)^K$ can be associated in the following way
\[
(X,A)^K=\bigcup_{I\in K}(X,A)^I
\]
where $(X, A)^{I}=\{(x_1, \ldots, x_m)\in \overset{m}{\underset{j=1}{\prod}}X\mid x_j\in A ~\mathrm{for}~j \notin I\}$.

In particular, when $(X,A)=(D^2, S^1)$, the polyhedral product $(D^2,S^1)^K$ is known as the moment-angle complex $\zk$. These objects and their topological and lately homotopy theoretical properties have been of main interest in toric topology.

If a finite group $G$ acts simplicially on a simplicial complex $K$, then that action induces a $G$-action on polyhedral products, in particular on the moment-angle complex $\zk$. Notice that by acting simplicially on a simplicial complex $K$ on $m$ vertices, $G$ is a subgroup of the symmetric group $\Sigma_m$.

In this paper we study $\Sigma_m$-representation stability of polyhedral products. We start by analysing $G$-equivariant properties of the stable homotopy decomposition of moment-angle complexes $\zk$~\cite{Ho, BP1} and polyhedral products $(X,A)^K$~\cite{BBCG}. These homotopy decompositions induce $\bk G$-module decompositions of the cohomology of moment-angle complexes and polyhedral products, respectively. Recall~\cite[Corollary 5.4]{KSB} that for a $G$-module $N\cong\oplus_{i\in I}N_i$, with the $G$-action permuting the summands of $N$ according to some $G$-action on $I$, there exists a $G$-isomorphism 
\begin{equation}\label{Brown}
N\cong \underset{i\in E}\oplus\mathrm{Ind}^G_{G_i}N_i
\end{equation}
where $E$ is a set of representatives of  orbits of $I$ and $G_i$ is the stabiliser of $i$ in $G$.

Specialising to $G=\Sigma_m$, we describe several non-trivial  constructions of families of simplicial $\Sigma_m$-complexes $\mathcal K=\{ K_m\}$ (see Constructions~\ref{fi2} and~\ref{fi3})
and describe conditions on these families which together with  decomposition~\eqref{Brown} and Hemmer's result~\cite{He} imply uniform representation stability of $\Sigma_m$-representation  of  $\{\widetilde{H}_*((X,A)^{K_m};\bk)\}$ (see Theorem~\ref{thmrrpp} and Corollary~\ref{rsconstructions}).

In the case of moment-angle complexes, we construct a sequence of $\Sigma_m$-manifolds which are uniformly representation stable although not homology stable (see Proposition~\ref{rsmanifolds}).

The uniform representation stability influences the behaviour of the Betti numbers of the $i$-th homology groups $\{ \widetilde{H}_i((X,A)^{K_m}; \Q)\}$
and we show that in this case their growth is eventually polynomial with respect to $m$ (see Theorem~\ref{Betti}).

\section{Moment-angle complexes associated with simplicial $G$-complexes}\label{sppg}

Moment-angle complexes $\zk=(D^2, S^1)^K$ are considered as spaces on which a torus $T^l$, $l\leq m $ acts. The action of the torus is induced by an $S^1$-action on $(D^2, S^1)$. Extensive literature exists on the study of this action. The problem we are studying is how symmetries of a simplicial complex $K$ influence the symmetries of the moment-angle complex $\zk$.

\subsection{$\bk G$-module structures on $H^*(\zk; \bk )$}

Let $\bk$ be a field or $\Z$, let $G$ be a finite group, and let $K$ be a simplicial $G$-complex. We shall describe $G$-actions on the moment-angle complex $\zk$ induced by a simplicial $G$-action on $K$.

A {\it $G$-complex} is a CW-complex $X$ together with a group action $G$ on it which permutes the cells. A {\it simplicial $G$-complex} is a simplicial complex $\K$ on a vertex set $[m]$ with a $G$-action on $[m]$ such that the induced action on subsets of $[m]$ preserves $\K$. Thus, the geometrical realisation of a simplicial $G$-complex $\K$ is a $G$-complex. For a simplicial $G$-complex $\K$, each chain group $C_n(\K; \bk)$ is a direct sum of copies of $\bk$, each summand corresponding to an $n$-simplex of $\K$ on which $G$ acts. Denote by $G_\sigma$ the stabiliser of $\sigma$, and let $E_n$ be a set of representatives of the $G$-orbits of $n$-simplices of $K$. Thus, by~\eqref{Brown}, 
\[
C_n(\K;\bk)\cong \underset{\sigma \in E_n} \bigoplus \mathrm{Ind} _{G_{\sigma}}^G \bk.
\]

A moment-angle complex $\zk$ can be given the following cellular decomposition. 
The disc $D^2$ has three cells $e^0, e^1$ and $e^2$ of dimensions 0, 1 and 2, respectively. The cells of $D^{2m}\cong (D^2)^m$ are parametrised by subsets $I, L \subseteq [m]$ with $I \cap L= \emptyset$, so that a cell denoted by $\kappa(L,I)$ is equal to $e_1 \times \ldots \times e_m$ in $D^{2m}$, where $e_i$ is the $2$-dimensional cell $e^2$ if $i \in I$, $e_i$ is the $1$-dimensional cell $e^1$ if $i \in L$, and $e_i$ is the point $e^0$ if $i \in [m]\setminus (I \cup L)$. Since $\zk$ is a subcomplex of $D^{2m}$ determined by the simplicial complex $K$, the cells of $\zk$ are those cells $\kappa(L,I)$, where $I \in \K$.

We start by showing that if $\K$ is a simplicial $G$-complex, the corresponding moment-angle complex $\zk$ is a $G$-complex. Let $2^{[m]}$ be the power set of $[m]$. Then the $G$-action on $\K$ can be extended to an action $\Phi$ on $2^{[m]}$. Specifically, $\Phi\colon G \times 2^{[m]} \longrightarrow 2^{[m]}$ is given by $\Phi(g, \{i_1, \ldots, i_l\})=\{g\cdot i_1, \ldots, g\cdot i_l\}$, where $g \in G$ and $\{i_1, \ldots, i_l\} \subset [m]$.

The simplicial $G$-action on $K$ induces a $G$-action on $\zk$, $\rho\, \colon G\times\zk\longrightarrow\zk$, through homeomorphisms of $\zk$ given by 
\begin{equation}\label{gaction}
\rho_g\cdot(z_1, \ldots, z_m)=(z_{g\cdot 1}, \ldots, z_{g\cdot m}).
\end{equation}

\begin{lemma}
For a simplicial $G$-complex $\K$, the moment-angle complex $\zk$ is a $G$-complex.
\end{lemma}
\begin{proof}
A cell $\kappa(L,I)$, $I \in \K$ of $\zk$ is mapped by $g\in G$ to $g\cdot\kappa(L,I)=\kappa(g\cdot L, g\cdot I)$ which is again a cell of $\zk$ as a simplicial $G$-action maps simplices to simplices and non-simplices to non-simplices. Thus, $\zk$ is a $G$-complex. 
\end{proof}

Geometralising the famous Hochster decomposition~\cite{Ho}, Buchstaber and Panov~\cite{BP1, Taras} together with Baskakov~\cite{baskakov} showed that $H^*(\mathcal{Z}_{K}; \bk )\cong \underset{J\subseteq [m]}\bigoplus\widetilde{H}^*(K_J; \bk)$ as $\bk$-algebras, where $K_J$ is  the full subcomplex of $K$ on $J$ defined by $K_J=\{\sigma\cap J\mid \sigma\in K\}$. We aim to show that this is a $\bk G$-algebra isomorphism.

\begin{lemma}\label{subcomplex}
Let $\K$ be a simplicial $G$-complex on $[m]$. Then for any subset $J \subseteq [m]$ and $g \in G$, the set $g\cdot\K_J=\{g\cdot\sigma\mid \sigma \in \K_J\}$ is the full subcomplex $\K_{g\cdot J}$. \end{lemma}
\begin{proof}
Since $\K_J$ is a subcomplex of $\K$, every subset $\tau$ of $\sigma$ is in $\K_J$ if $\sigma \in \K_J$. Hence for $\sigma\in K_J$, every subset $\tau'$ of $g\cdot\sigma$ is $g\cdot\tau$ for some $\tau\leq\sigma$ and therefore is in $g\cdot\K_J$. Thus $g\cdot\K_J$ is a subcomplex of $\K$. 

To check that $g\cdot\K_J$ is the full subcomplex $\K_{g\cdot J}$, we observe that $g\cdot\K_J=g\cdot(\K \cap J)=g\cdot\K \cap g\cdot J =\K \cap g\cdot J=\K_{g\cdot J}$.
\end{proof}

Denote by $\{i_0, \ldots, i_p\}$ an unoriented simplex in $\K$ and by $[i_0, \ldots, i_p]$ an oriented simplex in $\K$. For an oriented $p$-simplex $\sigma=[i_0, \ldots, i_p]$, let $\sigma^*=[i_0, \ldots, i_p]^*$ denote the basis cochain in $C^p(\K; \bk)$.

Next, we show that a simplicial $G$-action on $\K$ induces a $G$-action on $\underset{J\subseteq [m]}\bigoplus\widetilde{H}^*(K_J; \bk )$.

\begin{lemma}\label{mainlemma}
Let $\K$ be a simplicial $G$-complex. For every $g\in G$ and $J\subseteq [m]$,
\[
g\cdot\widetilde{H}^*(\K_J; \bk ) =\widetilde{H}^*(\K_{g\cdot J}; \bk ).
\]
\end{lemma}
\begin{proof}
Let $\sigma=[i_0, \ldots, i_{p}]$ be an oriented simplex in $\K_J$ and $\sigma^*$ be the corresponding base cochain in $C^p({\K_J};\bk)$. Since $g$ gives a bijection between the basis of $C^*(\K_J; \bk)$ and the basis of $C^*(\K_{g\cdot J}; \bk)$ by $\sigma^* \mapsto g\cdot\sigma^*=\epsilon(g, \sigma)(g\cdot\sigma)^*$, the cochain complex $C^*(\K_J; \bk)$ is isomorphic to $C^*(\K_{g\cdot J}; \bk)$ as abelian groups. As the coboundary operator $d$ is given by
\[ 
d \sigma^*=\sum\varepsilon_j \tau_j^*
\]
where the summation of the coboundary operator extends over all $(p+1)$-simplices $\tau_j$ having $\sigma$ as a face, and $\varepsilon_j=\pm 1$ is the sign with which $\sigma$ appears in the expression for $\partial \tau$, 
we obtain the commutative diagram
\[
\xymatrix{
C^*(\K_J;\bk )\ar[r]^{\cong} \ar[d]^{d}	&C^*(\K_{g\cdot J};\bk )\ar[d]^{d}\\
C^*(\K_J;\bk )\ar[r]^{\cong}	& C^*(\K_{g\cdot J};\bk ).
}
\]
Therefore $g$ induces an isomorphism between $\widetilde{H}^*(\K_J; \bk )$ and $\widetilde{H}^*(\K_{g\cdot J}; \bk )$.
\end{proof}


We continue by showing that the $G$-actions on $H^*(\zk; \bk)$ and $\underset{J\subseteq [m]}\bigoplus\widetilde{H}^*(K_J; \bk)$ are compatible.

On $\mc{C}^*(\zk; \bk)$ a multigrading can be defined. Consider a subset $J\subseteq [m]$ as a vector in $\mathbb{N}^m$ whose $j$-th coordinate is $1$ if $j\in J$, or is $0$ if $j\notin J$. Define a $\Z \oplus \mathbb{N}^m$-grading on $\mc{C}^*(\zk; \bk)$ as 
\[
\mc{C}^*(\zk; \bk)= \underset{J\subseteq [m]}{\bigoplus} \mc{C}^{*, 2J}(\zk; \bk)
\]
where $\mc{C}^{*, 2J}(\zk; \bk)$ is the subcomplex spanned by cochains $\kappa(J\setminus I, I)^*$ with $I\subseteq J$ and $I\in \K$ whose multidegree is $\mathrm{mg} \kappa(J\setminus I, I)^*=(-|J\setminus I|, J)$.

Buchstaber and Panov~\cite[Theorem 3.2.9]{BP2} showed that there are isomorphisms between $\widetilde{H}^{p-1}(\K_J; \bk)$ and $H^{p-|J|, 2J}(\zk; \bk)$ which are functorial with respect to simplicial maps and are induced by the cochain isomorphisms $f_J \colon C^{p-1}(\K_J; \bk) \to \mc{C}^{p-|J|, 2J}(\zk; \bk)$ given by
\begin{equation}\label{fJ}
f_J(\sigma^*)=\epsilon(\sigma, J)\kappa(J\setminus \sigma, \sigma)^*
\end{equation}
where $\sigma\in\K_J$ and $\epsilon(\sigma, J)=\underset{j\in \sigma}{\prod} \epsilon(j, J)$ with $\epsilon(j, J)=(-1)^{r-1}$ if $j$ is the $r$-th element of $J$.


The functorial property induces a commutative diagram  
\[
\xymatrix{
C^{p-1}(\K_J; \bk)\ar[r]^-{f_J} \ar[d]^{g}	&\mc{C}^{p-|J|, 2J}(\zk; \bk)\ar[d]^{g}\\
C^{p-1}(\K_{g\cdot J}; \bk)\ar[r]^-{f_{g\cdot J}}	& \mc{C}^{p-|g\cdot J|, 2g\cdot J}(\zk; \bk)
}
\]
 implying the following statement.

\begin{lemma} \label{gcochainiso}
If $\K$ is a simplicial $G$-complex, then $\mathcal{C}^*(\zk; \bk)$ is multigraded isomorphic to $\underset{J\subseteq [m]}{\bigoplus} C^{*}(\K_J; \bk)$
as $\bk G$-modules.
\end{lemma}

Passing to cohomology, we obtain the following corollary.

\begin{corollary}\label{galgeiso}
For a simplicial $G$-complex, $H^*(\zk; \bk )$ is isomorphic to $\underset{J\subseteq [m]}\bigoplus\widetilde{H}^*(K_J;\bk )$ as $\bk G$-algebras.
\end{corollary}
\begin{proof}
By \cite[Theorem 4.5.8]{BP2}, the multiplication on $\underset{J\subseteq [m]}\bigoplus\widetilde{H}^*(K_J;\bk )$ is given by 
\[
H^i(K_I; \bk)\otimes H^j(K_J; \bk) \to H^{i+j}(K_{I\cup J};\bk)
\]
which is induced by the simplicial inclusions  $K_{I\cup J}\to K_I \ast K_J$ for $I\cap J=\emptyset$ and zero otherwise.
Under this multiplication, the maps $f_J$ induce a $\bk$-algebraic isomorphism $\underset{J\subseteq [m]}\bigoplus\widetilde{H}^*(K_J;\bk) \to H^*(\zk; \bk)$. Since $f_{g\cdot J}\circ g=g\circ f_J$, the maps $f_J$ induce a $\bk G$-algebraic isomorphism.
\end{proof}

Now we state the main result of this section.

\begin{proposition} \label{grepi}
Let $\K$ be a simplicial $G$-complex. Then there are $\bk G$-algebra isomorphisms
\[
H^*(\zk; \bk) \cong 
\underset{J\in [m]/G} \bigoplus~ \underset{g\in G/{G_J}} \bigoplus g\cdot \widetilde{H}^{*-|J|-1}(K_J; \bk)
\]
where $G_J=\{g\in G\mid g\cdot J=J\}$ is the stabiliser of $J$ and $[m]/G$ is a set of representatives of $G$-orbits of $2^{[m]}$.

The multiplication on $\underset{J\in [m]/G} \bigoplus~ \underset{g\in G/{G_J}} \bigoplus g\cdot \widetilde{H}^*(K_J; \bk)$ is given so that for any $I, J \in [m]/G$ and $g\in G/G_J,~ h\in G/G_I$, there is a map 
\begin{center}
\scalebox{0.9}{$
\mu\colon g\cdot H^{k-|J|-1}(\K_J; \bk) \otimes h\cdot H^{l-|I|-1}(\K_I; \bk)=H^{k-|J|-1}(\K_{g\cdot J}; \bk) \otimes H^{l-|I|-1}(\K_{h\cdot I}; \bk) \rightarrow H^{k+l-|I|-|J|-1}(\K_{g\cdot J\cup h\cdot I}; \bk)
$}
\end{center}
which is induced by the simplicial inclusion $\K_{g\cdot J\cup h\cdot I}\longrightarrow \K_{g\cdot J}*\K_{h\cdot I}$ if $g\cdot J\cap h\cdot I=\emptyset$ and is a zero map otherwise.
\end{proposition}
\begin{proof}
Since by Corollary \ref{galgeiso} $H^*(\zk; \bk) \cong \underset{J\subseteq [m]}\bigoplus\widetilde{H}^*(K_J;\bk)$ as $\bk G$-algebras,
it suffices to show that the $G$-isomorphism
\[
\underset{J\subseteq [m]}\bigoplus\widetilde{H}^*(K_J;\bk ) \cong \underset{J\in [m]/G} \bigoplus~ \underset{g\in G/{G_J}} \bigoplus g\cdot \widetilde{H}^{*}(K_J; \bk)
\]
preserves the multiplications on both sides.
The multiplication on $\underset{J\in [m]/G} \bigoplus~ \underset{g\in G/{G_J}} \bigoplus g\cdot \widetilde{H}^*(K_J; \bk)$ is induced by the multiplication on $\underset{J\subseteq [m]}\bigoplus\widetilde{H}^*(K_J; \bk)$ via the above $G$-isomorphism. Therefore, 
\[
H^*(\zk; \bk) \cong 
\underset{J\in [m]/G} \bigoplus~ \underset{g\in G/{G_J}} \bigoplus g\cdot \widetilde{H}^{*-|J|-1}(K_J;\bk)
\]
as $\bk G$-algebras.
\end{proof}

We illustrate Proposition~\ref{grepi} on several examples.
\begin{example}
Let $\K$ be the boundary of a square, 
\begin{tikzpicture} 
\tikzstyle{point}=[circle,thick,draw=black,fill=black,inner sep=0pt,minimum width=2pt,minimum height=0pt] 
\node (1)[point, label=left:\small{${1}$}] at (0,-0.5) {}; 
\node (2)[point, label=right:\small{$2$}] at (0.5,-0.5) {}; 
\node (3)[point, label=right:\small{$3$}] at (0.5,0) {};
\node (4)[point, label=left:\small{$4$}] at
(0,0) {};
\draw (1) -- (2) -- (3) -- (4) -- (1); 
\end{tikzpicture}. 
It is a simplicial $C_4$-complex, where $C_4$ is the cyclic group of order $4$. Write $C_4=\{(1), (1234), (13)(24), (1432)\}$ as a subgroup of the permutation group $\Sigma_4$. A set of representatives of $2^{[4]}$ under $C_4$ is given by
\[
E=\{\emptyset, \{1\}, \{1,2\}, \{1,3\}, \{1,2,3\}, \{1,2,3,4\}\}.
\]

Taking $J$ to be an element in $E$, observe that
\[\widetilde{H}^{p}(\K_J; \bk)=
\begin{cases}\begin{array}{ll}
\bk & \text{when} ~J=\emptyset ~\text{and}~p=-1 \\
\bk &\text{when} ~J=\{1, 3\} ~\text{and}~p=0\\
\bk & \text{when} ~J=\{1, 2, 3, 4\} ~\text{and}~p=1\\
0& \text{otherwise}.
\end{array}
\end{cases}
\]
The stabilisers $G_J$ corresponding to $J=\emptyset$, $J=\{1, 3\}$ and $J=\{1, 2, 3, 4\}$ are $G_{\emptyset}=C_4$, $G_{13}=\{(1), (13)(24)\}$ and $G_{1234}=C_4$, respectively.
Therefore, the cohomology groups of $\zk$ are given by
\[H^i(\zk; \bk)=
\begin{cases}
\begin{array}{ll}
\bk \oplus \bk &\text{for}~ i=3\\
\bk &\text{for}~ i=0, 6.
\end{array}
\end{cases}
\]
\end{example}



\begin{example}\label{skeleton}
Let $\K=\Delta_{m}^k$ be the full $k$-skeleton of $\Delta^{m-1}$ which consists all subsets of $[m]$ with cardinality at most $k+1$.  The permutation group $\Sigma_m$ acts on $\K$ simplicially. A set of representatives of $2^{[m]}$ under the action of $\Sigma_m$ can be also chosen as 
\[
E=\{\emptyset, \{1\}, \{1, 2\}, \ldots, \{1, \ldots, m\}\}.
\]
For any $J=\{1, 2, \ldots, |J|\} \in E$, the stabiliser of $J$ is the Young subgroup $\Sigma_{|J|} \times \Sigma_{m-|J|}$. If $J \in E$ with $|J|\leq k+1$, then $\K_J =\Delta^{|J|-1}$. Thus $\widetilde{H}^*(\K_J; \bk)=0$.

If $J \in E$ with $k+2 \leq |J|\leq m$, then $\K_J$ is the full $k$-skeleton of $\Delta^{|J|-1}$. Recall that $\widetilde{H}^*(\K_J; \bk)=\underset{c}{\oplus}\bk $, where $c=\binom{|J|-1}{k+1}$ if $*=k$; otherwise $\widetilde{H}^*(\K_J; \bk)=0$. Therefore,

\[H^i(\zk; \bk)=
\begin{cases}\begin{array}{ll}
\bk &\text{where} ~i=0 \\
\underset{c}{\oplus} \bk &\text{where}~ c=\binom{m}{|J|} \binom{|J|-1}{k+1} ~\text{and}~ i=|J|+k+1\\
0&\text{otherwise}.
\end{array}
\end{cases}
\]
Let us remark that for $k=0$, the simplicial complex $K$ consists of $m$ disjoint points and denote by $\mathcal{Z}_m$ the moment-angle complex corresponding to it. 
By Proposition~\ref{grepi}, 
$H^3(\mathcal{Z}_m; \bk)$ has a basis $\{a_{ij}\mid 1\leq i<j\leq m\}$ and identifying $a_{ji}=-a_{ij}$, the symmetric group $\Sigma_m$ acts on $H^3(\mathcal{Z}_m; \bk)$ by a permutation of the indices.

For $K_m=\Delta_{m}^k$ with  $k$ fixed and  $m$ increasing, we get a sequence of moment-angle complexes $\{\mathcal{Z}_{K_{m}}\}$. There exist retractions $p_m\colon \mathcal{Z}_{K_{m+1}}\longrightarrow \mathcal{Z}_{K_{m}}$ obtained by restricting the projection map \scalebox{0.9}{$(D^2)^{m+1} \longrightarrow (D^2)^{m}$} to $\zk$. We shall consider the  representation stability of the sequence \scalebox{0.9}{$\{H^i(\zk; \bk), p_m^i\}$}  in Section~\ref{rspp}.
\end{example}

\section{Polyhedral products associated with simplicial $G$-complexes}
Moment-angle complexes are specific examples of polyhedral products $(X,A)^K$ which are constructed from combinatorial information of a simplicial complex $K$ and a topological pair $(X,A)$. Our next aim is to study symmetries of polyhedral products induced by the symmetries of $K$.
The geometric and homological properties of polyhedral products arising from simplicial $\Aut(K)$-complexes have been studied by Ali Al-Raisi in his PhD thesis~\cite{Ali}.
Al-Raisi proved that the map $(X, A)^K \longrightarrow \Omega\Sigma(\underset{I\subseteq [m]}{\bigvee} (X, A)^{\wedge \K_I})$ is homotopy $\Aut(K)$-equivariant.

In this section, we will give a different method for studying homotopy $G$-decompositions of polyhedral product $(X, A)^K$ associated with a simplicial $G$-complex $K$ by studying  the adjoint of the Al-Raisi map, known as the Bahri-Bendersky-Cohen-Gitler (BBCG) map, after several suspensions. We start with the BBCG homotopy decomposition for polyhedral products $(X, A)^{\K}$ (see~\cite{BBCG}). 
For any subset $I=\{i_1, \ldots, i_l\} \subseteq [m]$, and a pair of connected based CW-complexes $(X, A)$, recall the following notation
\[
\begin{split}
(X, A)^{I}&=\{(x_1, \ldots, x_m)\in \overset{m}{\underset{j=1}{\prod}}X\mid x_j\in A ~\mathrm{for}~j \notin I\},\\
(X, A)^{\wedge I}&=\{x_1\wedge\ldots \wedge x_m\in X \wedge \ldots \wedge X \mid x_j \in A ~\mathrm{for}~ j\notin I\},\\
X^{\wedge I}&=X_{i_1} \wedge \ldots \wedge X_{i_l}, ~\text{where~ each}~ X_{i_j}=X.
\end{split}
\]
For  a simplicial complex $K$ on $m$ vertices, the {\it polyhedral product} with respect to $(X, A)$ is defined as 
\[
(X, A)^{\K}=\underset{I\in \K}\bigcup(X, A)^I.
\]

Analogously, the {\it polyhedral smash product} of a topological pair $(X,A)$ and a simplicial complex $K$  is defined as 
\[(X, A)^{\wedge \K}=\underset{I \in \K}{\bigcup}(X, A)^{\wedge I}.\]

In~\cite{BBCG}, it was shown that the classical homotopy equivalence $\Sigma(X_1\times \ldots \times X_m)\to \Sigma\underset{I\subseteq [m]}\bigvee X^{\wedge I}$ induces the following homotopy decomposition.
\begin{equation}
\label{decmpoly}
\Sigma (X,A)^{\K} \simeq \Sigma \underset{I\subseteq [m]}{\bigvee} (X, A)^{\wedge \K_I}
\end{equation}
when $(X, A)$ is a topological pair of connected and based CW-complexes.


If $\K$ is a simplicial $G$-complex, then the $G$-action on $\K$ induces a cellular $G$-action on the corresponding polyhedral product $(X, A)^{\K}$ with respect to a pair of CW-complexes $(X, A)$, $A\subseteq X$. Explicitly, for $\underline{x}=(x_1, \ldots, x_m)\in (X, A)^{\K}$, $g\cdot (x_1, \ldots, x_m)=(x_{g\cdot 1}, \ldots, x_{g\cdot m})$. Thus $(X, A)^{\K}$ is a $G$-complex.

If $Y$ is a $G$-CW-complex, then each $i$-th homology group $H_i(Y; R)$ is an $RG$-module. Consider a natural $G$-action on $\Sigma Y$ by $g\cdot (\langle y, t\rangle)=\langle g\cdot y, t\rangle$ for $g\in G$. The naturality of long exact sequence for the topological pair $(CY, Y)$ implies that the isomorphism $H_{i+1}(\Sigma Y; R)\cong \widetilde{H}_i(Y; R)$ is an $RG$-isomorphism.

Consider  $X^{m}$ as a $\Sigma_m$-space given by
 $g\cdot  \underline{x}=g\cdot (x_1, \ldots, x_m)=(x_{g\cdot 1}, \ldots, x_{g\cdot m})$ for $g\in \Sigma_m$ and $x_i\in X$.
There exists a $\Sigma_m$-action on the based spaces $\Sigma X^{m}$ and $\Sigma (\underset{I\subseteq [m]}{\bigvee}X^{\wedge I})$, where $I$ runs over the non-empty subset of $[m]$. Explicitly, for every $g\in \Sigma_m$ and $\langle \underline{x}, t\rangle \in \Sigma X^{m}$, 
$g \cdot \langle \underline{x}, t\rangle =\langle g\cdot \underline{x}, t\rangle$.
 For any non-empty subset $I=\{i_1, \ldots, i_l\} \subseteq [m]$, each map $g\colon \Sigma X^{\wedge I} \to \Sigma X^{\wedge g\cdot I}$ sending $\langle x_{i_1}\wedge \ldots \wedge x_{i_l}, t\rangle$ to $\langle x_{g\cdot i_1}\wedge \ldots \wedge x_{g \cdot i_l}, t\rangle$ induces a $\Sigma_m$-action on $\Sigma \underset{I\subseteq [m]}{\bigvee}X^{\wedge I}$.

\begin{lemma}\label{symmetrichmpy}
There exists a homotopy equivalence
\[
\Sigma\theta_m\colon \Sigma^2 X^{m} \longrightarrow \Sigma^2 \underset{I\subseteq [m]}{\bigvee}X^{\wedge I}
\] 
that is $\Sigma_m$-equivariant.
\end{lemma}


\begin{proof}
For a non-empty set $I=\{i_1, \ldots, i_l\}\subseteq [m]$, define maps $\Sigma p^{\wedge I}$  by
\[
\begin{split}
\Sigma p^{\wedge I}\colon \Sigma X^{ m} &\longrightarrow \Sigma X^{\wedge I} \\
\langle x_1, \ldots, x_m, t\rangle &\longmapsto \langle x_{i_1}\wedge \ldots \wedge x_{i_l}, t\rangle
\end{split}
\]
Let $L=2^m-1$. Define a  comultiplication map $\delta_m \colon \Sigma X^{ m} \longrightarrow \overset{L}{\underset{j=1}\bigvee} \Sigma X^{m}$ on $\Sigma X^{m}$ such that \\
if $t\in [\frac{i}{L}, \frac{i+1}{L}]$ ($0\leq i\leq L-1$), 
\[
\delta_m(\langle x_1, \ldots, x_m, t\rangle)=(*, \ldots, *, \langle x_1, \ldots, x_m, Lt-i\rangle, *, \ldots, * )
\] 
where $\langle x_1, \ldots, x_m, Lt-i\rangle$ is in the $(i+1)$-st wedge summand of $\overset{L}{\underset{j=1}\bigvee} \Sigma X^{ m}$.

Fix an order $I_1 > I_2 > \cdots > I_{L}$ on the finite set $\{I\subseteq [m]\mid I\neq \emptyset\}$. Let each $I_j$ contain elements written in an increasing order.
Rewrite $ \Sigma (\underset{I\subseteq [m]}{\bigvee}X^{\wedge I})$ as $\Sigma X^{\wedge I_1}\vee \ldots \vee \Sigma X^{\wedge I_L}$.

Consider a map $\underset{I\in 2^{[m]}\setminus \emptyset}\bigvee \Sigma p^{\wedge I}\colon  \overset{L}{\underset{j=1}\bigvee} \Sigma X^{m} \longrightarrow   \Sigma (\underset{I\subseteq [m]}{\bigvee}X^{\wedge I})$ given by
\[
\underset{I\in 2^{[m]}\setminus \emptyset}\bigvee \Sigma p^{\wedge I}= \Sigma p^{\wedge I_1} \vee \ldots \vee \Sigma p^{\wedge I_{L}} \colon \overset{L}{\underset{j=1}\bigvee} \Sigma X^{m} \longrightarrow \Sigma X^{\wedge I_1}\vee \ldots \vee \Sigma X^{\wedge I_L}.
\]
Thus the map
\[
\theta_m=\underset{I\in 2^{[m]}\setminus \emptyset}\bigvee \Sigma p^{\wedge I} \circ \delta_m.
\]

Let $g\in \Sigma_m$ and $\langle x_1, \ldots, x_m, t\rangle \in \Sigma X^{m}$. For $t\in [\frac{i}{L}, \frac{i+1}{L}]$ ($0\leq i\leq L-1$), there is
\[
\begin{split}
\theta_m \circ g(\langle x_1, \ldots, x_m, t\rangle)&=\underset{I\in 2^{[m]}\setminus \emptyset}\bigvee \Sigma p^{\wedge I} \circ \delta_m (\langle x_{g\cdot 1}, \ldots, x_{g\cdot m}, t\rangle )\\
&=\underset{I\in 2^{[m]}\setminus \emptyset}\bigvee \Sigma p^{\wedge I}(*, \ldots, *,\underbrace{\langle x_{g\cdot 1}, \ldots, x_{g\cdot m}, Lt-i\rangle}_{i+1}, *,\ldots, *)\\
&=(*, \ldots, *,\underbrace{\langle x_{g\cdot m_1^{(i+1)}}\wedge \ldots\wedge x_{g\cdot m_s^{(i+1)}}, Lt-i\rangle}_{i+1}, *,\ldots, *)
\end{split}
\]
where $I_{i+1}=\{m_1^{(i+1)}, \ldots, m_s^{(i+1)}\}$ with $m_1^{(i+1)}< \cdots< m_s^{(i+1)}$.

Recall that $\underset{I\in 2^{[m]}\setminus \emptyset}\bigvee \Sigma p^{\wedge I}= \Sigma p^{\wedge I_1} \vee \ldots \vee \Sigma p^{\wedge I_{L}}$ and define by $\underset{I\in 2^{[m]}\setminus \emptyset}\bigvee \Sigma p^{\wedge (g\cdot I)}= \Sigma p^{\wedge (g\cdot I_1)} \vee \ldots \vee \Sigma p^{\wedge (g\cdot I_{L})}$.
Hence, $\theta_m\circ g=\underset{I\in 2^{[m]}\setminus \emptyset}\bigvee \Sigma p^{\wedge (g\cdot I)}\circ \delta_m$.

On the other hand, there exits a permutation $T$ of summand
$\overset{L}{\underset{j=1}\bigvee} \Sigma X^{m}$ induced by $g$ such that $g\circ \theta_m=\underset{I\in 2^{[m]}\setminus \emptyset}\bigvee \Sigma p^{\wedge (g\cdot I)}\circ T\circ \delta_m$.
Since $g$ acts on a set $\{1, \ldots, L\}$ by  $g\cdot i$ being the unique number satisfying $I_{g\cdot i}=g\cdot I_i$ as sets, this action on $\{1, \ldots, L\}$  induces a permutation $T$ of $\overset{L}{\underset{j=1}\bigvee} \Sigma X^{m}$.
Note that for $t\in [\frac{i}{L}, \frac{i+1}{L}]$ ($0\leq i\leq L-1$),
\[\begin{split}
&\underset{I\in 2^{[m]}\setminus \emptyset}\bigvee \Sigma p^{\wedge (g\cdot I)}\circ T\circ \delta_m (\langle x_{ 1}, \ldots, x_{ m}, t\rangle)\\
=&\underset{I\in 2^{[m]}\setminus \emptyset}\bigvee \Sigma p^{\wedge (g\cdot I)}\circ T(*, \ldots, *, \underbrace{\langle x_1, \ldots, x_m, Lt-i\rangle}_{i+1}, *, \ldots, * )\\
=&\underset{I\in 2^{[m]}\setminus \emptyset}\bigvee \Sigma p^{\wedge (g\cdot I)} (*, \ldots, *,\underbrace{\langle x_1, \ldots, x_m, Lt-i\rangle}_{g\cdot (i+1)}, *, \ldots, * )\\
= &(*, \ldots, *,\underbrace{\langle x_{g\cdot m_1^{(i+1)}}\wedge \ldots\wedge x_{g\cdot m_s^{(i+1)}}, Lt-i\rangle}_{g\cdot (i+1)}, *,\ldots, *)
\end{split}
\]
where $I_{i+1}=\{m_1^{(i+1)}, \ldots, m_s^{(i+1)}\}$ with $m_1^{(i+1)}< \cdots< m_s^{(i+1)}$.

Also, for $t\in [\frac{i}{L}, \frac{i+1}{L}]$ ($0\leq i\leq L-1$),
\[
\begin{split}
g\circ \theta_m (\langle x_{ 1}, \ldots, x_{ m}, t\rangle)&=g(*, \ldots, *, \underbrace{\langle x_{m_1^{(i+1)}}\wedge \ldots\wedge x_{m_s^{(i+1)}}, Lt-i\rangle}_{i+1}, *, \ldots, * )\\
&=(*, \ldots, *, \underbrace{\langle x_{g\cdot m_1^{(i+1)}}\wedge \ldots\wedge x_{g\cdot m_s^{(i+1)}}, Lt-i\rangle}_{g\cdot (i+1)}, *, \ldots, * ).
\end{split}
\]
Thus we have $g\circ \theta_m=\underset{I\in 2^{[m]}\setminus \emptyset}\bigvee \Sigma p^{\wedge (g\cdot I)}\circ T\circ \delta_m$.

Since $\Sigma \delta_m$ is cocommutative, $\Sigma (g\circ \theta_m) \simeq \Sigma (\theta_m\circ g)$.
\end{proof}

The following statement is a consequence of Lemma \ref{symmetrichmpy}.
\begin{lemma}\label{symmetrichmpy2}
a) For $g\in \Sigma_m$ and $I \subseteq [m]$, there is the homotopy commutative diagram
\begin{equation} \label{symmetricaction1}
\xymatrix{
\Sigma^2 (X, A)^I \ar[rr]^-{\simeq} \ar[d]^{g}	&& \underset{J\subseteq [m]}{\bigvee}\Sigma^2 (X, A)^{\wedge (I\cap J)} \ar[d]^{g}\\
\Sigma^2(X, A)^{g\cdot I}\ar[rr]^-{\simeq}	&& \underset{g\cdot J\subseteq [m]}{\bigvee} \Sigma^2 (X, A)^{\wedge g\cdot (I\cap J)}
}
\end{equation}
where the vertical map $g$ on the left is given by 
\[
g\cdot \langle x_1, \ldots, x_m, t, s\rangle=\langle x_{g\cdot1}, \ldots, x_{g\cdot m}, t, s\rangle 
\]
and the vertical map $g$ on the right maps each element in $\Sigma^2(X, A)^{\wedge (I\cap J)}$ into the corresponding one in $\Sigma^2(X, A)^{\wedge g\cdot (I\cap J)}$ via a coordinate permutation by $g$.

b) For an inclusion $I_1\subseteq I_2 \subseteq [m]$, there is the  diagram
\begin{equation}\label{symmetricaction2}
\scalemath{0.85}{\xymatrix{
& \Sigma^2(X, A)^{I_1} \ar[ld] \ar[rr]^-{g} \ar'[d][dd]^-{\simeq}  & & \Sigma^2(X, A)^{g\cdot I_1}  \ar[dd]^{\simeq}\ar[ld]
\\
\Sigma^2(X, A)^{I_2} \ar[rr]^{\qquad g}\ar[dd]^{\simeq}
& & \Sigma^2(X, A)^{g\cdot I_2} \ar[dd]^(0.6){\simeq}
\\
&  \underset{g\cdot J\subseteq [m]}{\bigvee}\Sigma^2 (X, A)^{\wedge  ({I_1}\cap J)} \ar[ld] \ar'[r][rr]^-g
& & \underset{g\cdot J\subseteq [m]}{\bigvee}\Sigma^2 (X, A)^{\wedge g\cdot ({I_1}\cap J)} \ar[ld]
\\
\underset{J\subseteq [m]}{\bigvee} \Sigma^2 (X, A)^{\wedge ({I_2}\cap J)} \ar[rr]^-g
& & \underset{J\subseteq [m]}{\bigvee} \Sigma^2 (X, A)^{\wedge g\cdot ({I_2}\cap J)}
}}
\end{equation}
where the four side diagrams are homotopy commutative and the top and bottom diagrams are commutative.\qed 
\end{lemma}

Since the homotopy decomposition $\Sigma^2 (X, A)^{\K}\simeq \Sigma^2 \underset{J\subseteq [m]}{\bigvee}(X, A)^{\wedge \K_J}$ is natural with respect to inclusions in $\K$~(\cite[Theorem 2.10]{BBCG}), the next result follows immediately from the lemma above. 

\begin{theorem} \label{gppdec}
Let $\K$ be a simplicial $G$-complex with $m$ vertices. Then there is a homotopy $G$-decomposition
\begin{equation}\label{Geq}
\theta \colon \Sigma^2(X, A)^{\K} \simeq \Sigma^2\underset{J\subseteq [m]}{\bigvee}(X, A)^{\wedge\K_J}
\end{equation}
where the $G$-action on $\Sigma^2(X, A)^{\K}$ is induced by the $G$-action on $X^m$, and the $G$-action on the right hand side is induced by (\ref{symmetricaction1}).
\end{theorem}
\begin{proof}
Let $\mathrm{CAT}(K)$ be the face category of $K$ consisting of simplices of $K$ and simplicial inclusions in $K$. 
Define two functors $D$ and $E$ from $\mathrm{CAT}(K)$ to $CW_*$ by $D(\sigma)=(X, A)^{\sigma}$ and $E(\sigma)=\underset{J\subseteq [m]}\bigvee (X, A)^{\wedge (\sigma\cap J)}$ for $\sigma\in \mathrm{CAT}(K)$.
For every $\sigma\in \mathrm{CAT}(K)$ and $g\in G$, diagram (\ref{symmetricaction1}) implies that there exists a homotopy
\[
H_g(\sigma)\colon \Sigma^2 D(\sigma)\times \mathbb{I}  \longrightarrow\Sigma^2 E(g\cdot \sigma)
\]
such that $H_g(\sigma)(x, 0)=\theta(\sigma)$ and $H_g(\sigma)(x, 1)=\theta(g\cdot \sigma)$, where $\theta(\sigma)$ is the natural homotopy equivalence between $\Sigma^2 D(\sigma)$ and $\Sigma^2 E(\sigma)$ and $\mathbb{I}$ is the interval $[0, 1]$. 
Diagram (\ref{symmetricaction2}) implies that if $\sigma, \tau \in \mathrm{CAT}(K)$,
 then $H_g(\sigma \cap \tau)=H_g(\sigma)|_{\Sigma^2 D(\sigma\cap \tau)\times \mathbb{I}}=H_g(\tau)|_{\Sigma^2 D(\sigma\cap \tau)\times \mathbb{I}}$.
$H_g(\cdot)$ induces a natural transformation from $\Sigma^2 D(\cdot) \times \mathbb{I}$ to $\Sigma^2 E(\cdot)$.

With $g$ fixed, $H_g(\sigma)$ will induce a continuous map $H_g \colon \mathrm{colim}\,\Sigma^2 D \times \mathbb{I} \longrightarrow \mathrm{colim}\,\Sigma^2 E$ such that $H_g(x, 0)=g\theta(x)$ and $H_g(x, 1)=\theta(g\cdot x)$. Therefore, $\theta$ is a homotopy $G$-decomposition.
\end{proof}

\begin{example}
Let $K$ be the $k$-skeleton of a simplex $\Delta^{m-1}$ on which $\Sigma_m$ acts by permuting vertices. By Porter \cite{Porter}, Grbi\'{c}-Theriault \cite{GT1},
the homotopy type of $(\mathrm{Cone}A, A)^K$ is the wedge 
\[
(\mathrm{Cone}A, A)^K\simeq \underset{j=k+2}{\overset{m}\bigvee} \big( \underset{1\leq i_1< \ldots <i_j\leq m}\bigvee \binom {j-1} {k+1}\Sigma^{k+1}A_{i_1} \wedge\ldots \wedge A_{i_j} \big).
\]
Although  $\Sigma_m$ acts on both sides this homotopy equivalence might not be a homotopy $\Sigma_m$-equivalence. However after suspending it twice, by Theorem~\ref{gppdec} it is a homotopy equivariant map.
\end{example}

Considering $G$-equivalence~\eqref{Geq} and observing the induced $G$-actions on the reduced homology groups, we have the following result.

\begin{theorem}\label{greppp}
Let $\K$ be a simplicial $G$-complex on $m$ vertices. Then there exists a $\bk G$-module isomorphism 
\[
\widetilde{H}_i((X, A)^{\K}; \bk)\cong \underset{J\subseteq [m]}{\bigoplus}\widetilde{H}_i((X, A)^{\wedge\K_J}; \bk)\cong \underset{J\in [m]/G}{\bigoplus}\mathrm{Ind}_{G_J}^G \widetilde{H}_i((X, A)^{\wedge\K_J}; \bk)
\]
where $G$ acts on the middle term by permuting the summands such that $g\cdot \widetilde{H}_i((X, A)^{\wedge\K_J}; \bk)=\widetilde{H}_i((X, A)^{\K_{g\cdot J}}; \bk)$, $[m]/G$ is a set of  representatives of orbit of $2^{[m]}\setminus \emptyset$ under $G$ and $G_J$ is the stabiliser of $J$.\qed
\end{theorem}

\section{Representation stability for polyhedral products} \label{rspp}

Let $G$ be a finite group and $\bk$ be a field of characteristic zero. Then a $G$-action on a simplicial complex $K$ induces a $G$-complex structure on the corresponding polyhedral product $(X, A)^{K}$ and therefore its homology is a $\bk G$-module.
 Since every $\bk G$-module is a $G$-representation over $\bk$, we are able to use representation theory to study the homology groups of polyhedral products associated with simplicial $G$-complexes. 
Representation stability studies  a sequence of finite dimensional vector spaces such that each vector space $V_m$ is equipped with a $G_m$-action and each $V_m\overset{\psi_m}\longrightarrow V_{m+1}$ is $G_m$-equivariant.  
Here groups ${G_m}$ are not arbitrary; they all belong to a fixed family of groups whose $\bk$-linear irreducible representations are determined by some datum $\lambda$ which is independent of $G_m$ and therefore of $m$.  One such family consists of symmetric groups $\Sigma_m$, which we will consider in this section.
The idea of representation stability was firstly introduced by Church and Farb in~\cite[Section 2.3]{CF}.
Stability in representation theory generalises a classical homological stability. 
A sequence $\{Y_m\}$ of groups, manifolds or topological spaces with maps $Y_m \overset{\psi_m}\longrightarrow Y_{m+1}$
for each $i\geq 0$ is called homology stable if the map $H_i(Y_m) \overset{(\psi_m)_*}\longrightarrow H_i(Y_{m+1})$ is an isomorphism for a sufficiently large $m$.

We recall the precise definition of uniformly representation stability of representations of symmetric groups according to Church and Farb~\cite[Definition 2.6]{CF}. Let $\{V_m, \psi_m\}$ be a sequence of  $\Sigma_m$-representations so that the group $\Sigma_m$ acts on $V_{m+1}$ as a subgroup of $\Sigma_{m+1}$. Then it is \emph{consistent} if each $V_m$ decomposes as a direct sum of finite-dimensional irreducible representations. If given any partition $\lambda=(\lambda_1, \ldots, \lambda_l)\vdash k$, then for $m\geq \lambda_1+k$, the partition $\lambda[m]=(m-k, \lambda_1, \ldots, \lambda_l)$ is called \emph{padded partition}. Its corresponding irreducible representation is denoted by $V(\lambda)_m$.
 
Let now $\{V_m, \psi_m\}$ be a consistent sequence of $\Sigma_m$-representations over a field $\bk$ of characteristic 0. The sequence $\{V_m, \psi_m\}$ is \emph{uniformly representation stable} with stable range $m\geq N$ if each of the following conditions holds for all $m\geq N$.

1. Injectivity: The natural map $\psi_m\colon V_m \to V_{m+1}$ is injective.

2. Surjectivity: The $\Sigma_{m+1}$-orbit of $\psi_m(V_m)$ spans $V_{m+1}$.

3. Multiplicities (uniform): Decompose $V_m$ into irreducible representations as 
\[
V_m=\underset{\lambda}{\bigoplus} c_{\lambda, m} V(\lambda)_m
\]
with multiplicities $0\leq c_{\lambda, m}\leq \infty$. There is some $M$, not depending on $\lambda$, so that for $m\geq M$ the multiplicities $c_{\lambda, m}$ are independent of $m$ for all $\lambda$.

Hemmer~\cite{He} proved the uniform representation stability of $\Sigma_m$-representations $\{\mathrm{Ind}_{H\times \Sigma_{m-k}}^{\Sigma_m} V\boxtimes \bk\}$ induced by an $H$-representation  $V$, where $H\leq \Sigma_k$. Note that $V$ can be seen as an $(H \times \Sigma_{m-k})$-representation, where $\Sigma_{m-k}$ acts on $V$ trivially, denoted by $V\boxtimes \bk$  for any $m \geq k$.

In this section, we study the representation stability arising in polyhedral products over a sequence of finite simplicial $\Sigma_m$-complexes.

\begin{definition}
A sequence of finite simplicial complexes
\[
\K_0\subseteq\K_1 \subseteq \K_2 \subseteq \ldots \subseteq \K_m \subseteq \K_{m+1} \subseteq \ldots
\]
where $K_0=\emptyset$ and each $\K_m$ is a simplicial $\Sigma_m$-complex and the simplicial inclusion $i_m\colon \K_m\subseteq \K_{m+1}$ is $\Sigma_m$-equivariant ($\Sigma_m$ acts on $\K_{m+1}$ via $\Sigma_m\hookrightarrow \Sigma_{m+1}$) is called a \emph{consistent sequence}.
\end{definition}

We start by considering few families of consistent sequences of finite simplicial complexes. The main aim of the paper is to show that these consistent sequences induce the consistent sequence of $\Sigma_m$-representations of the homology of polyhedral products which are representation stable. 

\begin{example}\label{fi1}  ($k$-skeleton sequences)
Fix an integer $k\geq 0$. To each $m$ assign the $k$-skeleton $\Delta_m^k$ of a standard $(m-1)$-simplex,
\begin{equation}
\label{simplexsq}
\emptyset\subseteq \Delta_1^k \subseteq \ldots \subseteq \Delta_m^k \subseteq \Delta_{m+1}^k \subseteq \ldots.
\end{equation}
The action of $\Sigma_m$ on $K_m$ is induced by permutations of all $m$ vertices. Each simplicial inclusion $i_m\colon \Delta^k_m\longrightarrow\Delta^k_{m+1}$ is $\Sigma_m$-equivariant. Therefore~\eqref{simplexsq} is consistent.
\end{example}

In general, if $K$ and $L$ are simplicial $G$-complexes on $V(K)$ and $V(L)$ respectively, then the $G$-action can be extended to the join $K*L$, as a complex on $V(K)\cup V(L)$  vertices, diagonally. 
\begin{construction}\label{fi2}
Fix integers $s\geq 1$ and $k_1, \ldots, k_s \geq 0$. For each $m\geq 0$, let $\K_m$ be a simplicial complex on $sm$ vertices given by the join of $\Delta_m^{k_1}, \Delta_m^{k_2}, \ldots, \Delta_m^{k_s}$. Since each $\Delta_m^{k_i}$ is a simplicial $\Sigma_m$-complex, then $\K_m$ is also a simplicial $\Sigma_m$-complex with the $\Sigma_m$-action given by $g\cdot(\sigma_1\sqcup \ldots \sqcup \sigma_s)=g\cdot\sigma_1\sqcup \ldots \sqcup g\cdot\sigma_s$ for $g\in \Sigma_m$ and for each $\sigma_i\in \Delta_m^{k_i}$. Let us consider the sequence
\[
\emptyset\subseteq \Delta_1^{k_1}*\ldots*\Delta_1^{k_s} \subseteq \ldots \subseteq \Delta_m^{k_1}*\ldots*\Delta_m^{k_s} \subseteq\Delta_{m+1}^{k_1}*\ldots* \Delta_{m+1}^{k_s} \subseteq \ldots.
\]
 The inclusion $K_m\subseteq K_{m+1}$ is given as a join of coordinate $\Sigma_m$-equivariant inclusions $\Delta_m^{k_i} \subseteq \Delta_{m+1}^{k_i}$ and therefore it is $\Sigma_m$-equivariant.
 
 Notice that for $s=1$ we recover the family of $k$-skeleton sequences of Example~\ref{fi1}.
\end{construction}

Next we construct a non-tivial example of consistent sequence of finite simplicial $\Sigma_m$-complexes.
\begin{construction}\label{fi3}
Let $I^m$ be an $m$-cube. Consider the simplicial complex $\K_m$ obtained by taking the boundary of the dual of a simple polytope $vc(I^m)$, where $vc(I^m)$ is obtained by cutting a vertex from $I^m$. Note that $K_1$ consists of two disjoint points. $\K_m$ can also be constructed as follows. Let $S_{2m}=S_1^0 * \ldots *S_m^0$ be the join of $m$ copies of two disjoint points, where $S_i^0=\{0_i, 1_i\}$. Notice that $S_{2m}$ is a triangulation of an $(m-1)$-sphere on $2m$ vertices. Then $K_m$ is obtained from $S_{2m}$ by deleting the interior of the $(m-1)$-face on vertices $0_1, \ldots, 0_m$ and taking the cone on it. The natural $\Sigma_m$-action on $\K_m$ is given by $g\cdot 0_i=0_{g\cdot i}$, $g\cdot 1_i=1_{g\cdot i}$, and $g$ fixes the cone vertex.
The inclusions $K_m\subseteq K_{m+1}$ are induced by the inclusions $S_1^0 * \ldots *S_m^0\subseteq S_1^0 * \ldots *S_m^0*S_{m+1}^0$ with the cone vertex mapping to itself. 

For example, when $m=2$, $K_2$ is simplicially isomorphic to a pentagon, and the $\Sigma_2$-action on $K_2$ is given by $0_1$ mapping to $0_2$ and $1_1$ mapping to $1_2$ keeping the cone vertex fixed.
As shown in the picture below, the blue colour lines represent how $K_2$ is included into $K_3$. 
\end{construction}

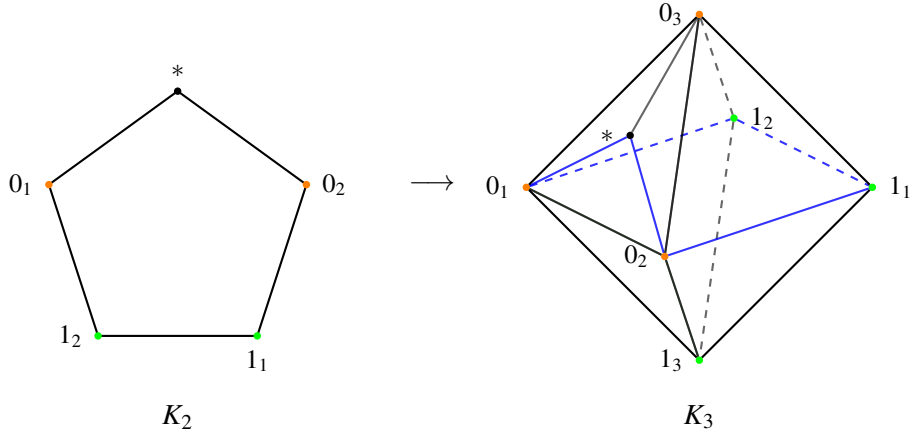
\begin{center}
\begin{tikzpicture}[scale=1.8]
\tikzstyle{point}=[circle,thick,draw=black,fill=black,inner sep=0pt,minimum width=2pt,minimum height=0pt]
\draw[thick] (0,1)--(-0.9510565163,0.309017)--(-0.58778525229,-0.809017)--(0.58778525229,-0.809017)--(0.9510565163,0.309017)--cycle;
\node [point, label=above:\small{$*$}] at (0,1) {};
\node [point,orange, label=right:\small{$0_2$}] at (0.9510565163,0.309017) {};
\node[label=right:$\longrightarrow$] at(1.54,0.309017){};
\node [point, green, label=below:\small{$1_1$}] at (0.58778525229,-0.809017) {};
\node [point, green, label=left:\small{$1_2$}] at (-0.58778525229,-0.809017) {};
\node [point, orange, label=left:\small{$0_1$}] at (-0.9510565163,0.309017) {};
\node[align=center, below] at (0,-1.25)%
{$K_2$}; 
\end{tikzpicture}
\begin{tikzpicture} [thick,scale=4.6]
\tikzstyle{point}=[circle,thick,draw=black,fill=black,inner sep=0pt,minimum width=2pt,minimum height=0pt] 
\node (a1)[point, orange, label=left:\small{$0_1$}] at (0,0) {}; 
\node (b2)[point, green, label=right:\small{$1_2$}] at (0.6,0.2) {}; 
\node (b1)[point, green, label=right:\small{$1_1$}] at (1,0) {};
\node (a2)[point, orange, label=left:\small{$0_2$}] at (0.4,-0.2) {};
\node (a3)[point, orange, label=left:\small{$0_3$}] at (0.5,0.5) {};
\node (b3)[point, green, label=left:\small{$1_3$}] at (0.5,-0.5) {};
\node (c)[point, label=left:\small{$*$}] at (0.3, 0.15) {};
\begin{scope}[thick,dashed,,opacity=0.6]
\draw[blue, opacity=0.8] (a1) -- (b2) -- (b1);
\draw[opacity=0.6] (a3) -- (b2) -- (b3);
\end{scope}
\draw[blue,opacity=0.8] (c)--(a1);
\draw[blue, opacity=0.8] (c)--(a2);
\draw[blue, opacity=0.8] (b1) -- (a2);
\draw[opacity=0.6] (c)--(a3);
\draw[opacity=0.6] (a2) -- (a3);
\draw[opacity=0.6] (a2) -- (b3);
\draw[opacity=0.6] (a1) -- (a2) -- (a3);
\draw[fill=green, opacity=0.6] (a1) -- (a2) -- (b3);
\draw (a3) -- (a1) -- (b3) -- (b1) --(a3); 
\node[align=center, below] at (0.5,-0.6)%
{$K_3$};
\end{tikzpicture} 
\begin{tikzpicture}
\node[label=right: $\text{\small{Vertices with the same color belong to the same orbit of symmetric actions.}}$] at(-0.3,-0.7){};
\end{tikzpicture}
\end{center}


\begin{definition}\label{facestable}
Given an integer $r\geq 0$, a consistent sequence $\mathcal K=\{\K_m, i_m\}$ of finite simplicial $\Sigma_m$-complexes is called \emph{$r$-face-stable} at degree $d$
if for $m\geq d$ and every $\sigma\in K_m$ with $\mathrm{dim} ~\sigma=r$ there exist a $g\in \Sigma_m$ and $\tau \in K_{d}$ such that $g\cdot i_{d, m}(\tau)=\sigma$, where $i_{d, m}=i_m\circ \ldots \circ i_d$ is a composite of the inclusions $i_d, \ldots, i_m$.

Similarly, a consistent sequence $\mathcal K=\{\K_m, i_m\}$ of finite simplicial $\Sigma_m$-complexes is called \emph{$r$-vertex-stable} at degree $d$
if for $m\geq d$ and any collection $\{v_0, \ldots, v_r\}$ of $r+1$ vertices of $K_m$ there exist a $g\in \Sigma_m$ and a collection $\{u_0, \ldots, u_r\}$ of $r+1$ vertices in $K_{d}$ such that $g\cdot i_{d, m}(u_i)=v_i$. In particular, if $r=0$ then $\mathcal K$ is called \emph{vertex-stable}.

If a consistent sequence $\{\K_m, i_m\}$ of finite simplicial $\Sigma_m$-complexes is $r$-vertex-stable (resp. $r$-face-stable) for every $r\geq 0$, we call it \emph{completely surjective} (resp. \emph{simplicially surjective}).
\end{definition}

\begin{bigremark}Note that Construction~\ref{fi2} and Construction~\ref{fi3} are completely surjective.

(i) Let $K_m=\Delta_{m}^k$. For $r\geq 0$, let $E_{m, r+1}$ consist of all the subsets of $[m]$ with cardinality $r+1$. Then the transitivity of  $\Sigma_m$-action on $E_{m, r+1}$ implies that $\{K_m\}$ is $r$-vertex-stable at degree $r+1$. 


(ii) Let $K_m=\Delta_m^{k_1} * \ldots *\Delta_m^{k_s}$.
If $s=2$, then $\Delta_m^{k_1}* \Delta_m^{k_2}$ is $r$-vertex-stable at degree $r+1$. Let $J_1, J_2$ be two subsets of $[m]$ with $|J_1|+|J_2|=r+1$ and $m\geq r+1$. If $J_1\cap J_2\neq \emptyset$, $J_1\cap J_2$ can be seen as a subset of vertices  of $\Delta_m^{k_1}$ and $\Delta_m^{k_2}$, respectively. Let $J_1^{c}=J_1\setminus J_1\cap J_2$ and $J_2^{c}=J_2\setminus J_1\cap J_2$ with cardinalities $r_1$ and $r_2$ and let $r_0=|J_1\cap J_2|$.

Define $g\in \Sigma_m$ by sending $\{1, \ldots, r_0\}$ to $J_1\cap J_2$, $\{r_0+1, \ldots, r_0+r_1\}$ to $J_1^c$ and $\{r_0+r_1+1, \ldots, r_0+r_1+r_2\}$ to $J_2^c$
and the complement of $\{1, \ldots, r_0+r_1+r_2\}$ in $[m]$ to the complement of $J_1\cup J_2$ in $[m]$, respecting to the initial order of vertices.

Now take the subset of vertices $\{1, \ldots, r_0+r_1\}$  of $\Delta_{r+1}^{k_1}$ and the subset of vertices $\{1, \ldots, r_0, r_0+r_1+1, \ldots, r_0+r_1+r_2  \}$ of $\Delta_{r+1}^{k_2}$
satisfying $g\cdot (\{1, \ldots, r_0+r_1\}\sqcup \{1, \ldots, r_0, r_0+r_1+1, \ldots, r_0+r_1+r_2 \})=J_1\sqcup J_2$.

If $J_1\cap J_2=\emptyset$, then $r_0=0$ and $g\in \Sigma_m$ sending $\{1, \ldots, r_1\}$ to $J_1$ and $\{r_1+1, \ldots, r_1+r_2\}$ to $J_2$
and the complement of $\{1, \ldots, r_1+r_2\}$ in $[m]$ to the complement of $J_1\cup J_2$ in $[m]$.  Inductively, $K_m$ is completely surjective.

(iii) For any $r\geq 0$, $K_m=\partial vc(I^m)^*$ is $r$-vertex-stable at degree $d=r+1$. With $m\geq d$, let $J$ be a subset of vertices of $K_m$ and $|J|=r+1$. Write $J=J_*\sqcup J_1 \sqcup \ldots \sqcup J_m$, where $J_*$ is either empty or the cone vertex $\{*\}$ and each $J_i\subseteq \{0_i, 1_i\}$. Since $|J|=r+1$, there are at most $r+1$ nonempty  components of $J$, say $J_{t_1}, \ldots, J_{t_{r+1}}$. If $*\notin J$, define $g\in \Sigma_m$ by sending $i$ to $t_i$ if $i\leq r+1$ and to $k_{i-r-1}$ otherwise where $\{k_1, \ldots, k_{m-r-1}\}$ is the complement of $\{t_1, \ldots, t_{r+1}\}$ in $[m]$. Now let $J^{\prime}=J_{1}^{\prime}\sqcup \ldots\sqcup J_{r+1}^{\prime}$ from the vertex set of $K_{r+1}$ where $J_{i}^{\prime}$ contains $0_i$ or $1_i$ if and only if $J_{t_i}$ contains $0_{t_i}$ or $1_{t_i}$. If $*\in J$, consider $\widetilde{J}=J\setminus \{*\}$ and repeat the above procedure to find $g\in \Sigma_m$ and $\widetilde{J}^{\prime}\in \mathrm{Ver}(K_{r+1})$ for $\widetilde{J}$. Then let $J^{\prime}=J_*\sqcup \widetilde{J}^{\prime}$ and $g\cdot J^{\prime}=J$.
\end{bigremark}



By Theorem~\ref{greppp}, for a simplicial $G$-complex K on $m$ vertices
\begin{equation}\label{summand}
\widetilde{H}_i((X, A)^{\K}; \bk)\cong  \underset{J\in [m]/G}{\bigoplus}\mathrm{Ind}_{G_J}^G \widetilde{H}_i((X, A)^{\wedge K_J}; \bk).
\end{equation}

If a consistent sequence $\{K_m,i_m\}$ of $\Sigma_m$-complexes  $K_m$ on the vertex set $V(K_m)$ is completely surjective then the summands in~\eqref{summand} do not depend on $m$ for sufficiently large $m$.  We shall use Hemmer's result to study the uniformly representation stability of polyhedral products. For that  the stabiliser $(\Sigma_m)_J$ needs to be of the form $H\times\Sigma_{m-k}$ for some $H\leq\Sigma_k$. Therefore we proceed by studying the stabiliser of $J\in\mathcal P(V(K_m))$ in $\Sigma_m$ which we denote by $\stab(J,m)$.

Observe that for a fixed integer $d$, for all $m\geq d$ and for some $J\in \mathcal{P}(V(\K_d))$, as $J$ also belongs to the $\Sigma_m$-set $\mathcal{P}(V(\K_m))$, there is a sequence of stabilisers 
\[
\scriptsize{\xymatrix{
\ldots \ar[r]  & \Sigma_m\ar[r]  &\Sigma_{m+1}\ar[r] &\ldots \\
\ldots \ar[r] & \stab(J, m)\ar[u]\ar[r]& \stab(J, m+1)\ar[u]\ar[r] &\ldots .
}}
\]
For instance, in Example~\ref{fi1}, if $m\geq |J|$, then $J\in \mathcal{P}(V(\Delta_m^k))$ and $\stab(J, m)=\Sigma_{|J|}\times \Sigma_{m-|J|}$. 

In Construction~\ref{fi2}, $K_{m}=\Delta_m^{k_1}\ast \Delta_m^{k_2}\ast \ldots\ast \Delta_m^{k_s}$. Write $J$ as a disjoint union of $J_1, \ldots, J_s$, where each $J_t$ ($1\leq t\leq s$) is from the $t$-th component $\mathcal{P}(V(\Delta_m^{k_t}))$. 
Let $b(J)=\underset{1\leq t\leq s}{\mathrm{max}}~|J_t|$. For $m\geq b(J)$, we observe the stabilisers of $J$ in $\Sigma_m$,
\[
\stab(J, m)=\{g\in \Sigma_{m}\mid g\cdot J_t=J_t, 1\leq t\leq s\}=\underset{1\leq t\leq s}{\bigcap}\stab(J_t, m)
\]
where, as in Example~\ref{fi1}, each $\stab(J_t, m)$ is isomorphic to $\Sigma_{|J_t|}\times \Sigma_{m-|J_t|}$.
For integers $a\leq b\leq m$, we have
\[
\Sigma_a\times \Sigma_{m-a} \cap \Sigma_b\times \Sigma_{m-b}= \Sigma_a \times \Sigma_{b-a} \times \Sigma_{m-b}.
\]
Therefore $\stab(J, m)=\stab(J, b(J))\times \Sigma_{m-b(J)}$ for $m\geq b(J)$, and $\Sigma_{m-b(J)}$ acts on $J$ trivially.


We call such sequences \emph{stabiliser consistent}.
\begin{definition}\label{nicestable}
A consistent sequence $\mathcal K=\{\K_m, i_m\}$ of finite simplicial $\Sigma_m$-complexes is called \emph{stabiliser consistent} if for every $d$ and every finite set $J\in \mathcal{P}(V( K_d))$ there exists an integer $b(J)$, such that if $m\geq d$ and $m\geq b(J)$, then either $\Sigma_m$ acts on $J$ trivially or the stabiliser $\stab(J, m)$ is isomorphic to $\stab(J, b(J))\times \Sigma_{m-b(J)}$, where $\Sigma_{m-b(J)}$ acts on $J$ trivially.
\end{definition}

Construction~\ref{fi3} also provides a stabiliser consistent sequence. 
If $J\in \mathcal{P}(V(K_d))$ for some $d$, write $J=J_{\ast} \sqcup J_1 \sqcup \ldots \sqcup J_d$, where $J_{\ast}$ is either empty or $\{\ast\}$ and each $J_i\subseteq \{0_i, 1_i\}$. Since $\Sigma_m$ acts on $\ast$ trivially, $\stab(J, m)=\stab(\tilde{J}, m)$ where $\tilde{J}=J_1 \sqcup \ldots \sqcup J_d$.
Let $b(J)$ be the number of non-empty components $J_t$. Then for $m\geq b(J)$, $\stab(J, m)\cong \stab(J, b(J))\times \Sigma_{m-b(J)}$ where $\Sigma_{m-b(J)}$ acts on $J$ trivially.

As a consequence, we have the following result that states conditions on a sequence of finite simplicial complexes that will induce in homology a uniformly representation stable sequence.
\begin{theorem}\label{thmrrpp}
Let $\{\K_m, i_m\}$ be a consistent sequence of finite simplicial complexes and $X$ be a connected, based $CW$-complexes of finite type with a based subcomplex $A$.
Suppose that $\{\K_m, i_m\}$ is completely surjective and stabiliser consistent.
Then the consistent sequence of $\Sigma_m$-representations $\{\widetilde{H}_i((X, A)^{\K_m}; \bk), i_{m_*}\}$ for $\mathrm{char}\bk=0$ is uniformly representation stable.
\end{theorem}
\begin{proof}
By Theorem~\ref{greppp}, we have
\begin{equation}\label{rrpp}
\widetilde{H}_i((X, A)^{\K_m}; \bk)\cong \underset{J\in E_m}{\bigoplus}\mathrm{Ind}_{\stab(J, m)}^{\Sigma_m} \widetilde{H}_i((X, A)^{\wedge \K_{m, J}}; \bk)
\end{equation}
where $E_m$ is a set of representatives of $\mathcal{P}(V(\K_m))$ under the action $\Sigma_m$, and $\stab(J, m)$ is the stabiliser of $J$ under $\Sigma_{m}$.

We prove that if $|J|\geq i+1$ then $\widetilde{H}_i((X, A)^{\wedge \K_{m, J}}; \bk)$ is trivial. By the reduced K\"{u}nneth formula for path-connected spaces, it is obvious that $\widetilde{H}_i(Y_1\wedge \ldots \wedge Y_{|J|}; \bk)=0$ if $|J|\geq i+1$, where each $Y_i$ is either $X$ or $A$.
This implies that for any $\sigma\in \K_{m, J}$, $\widetilde{H}_i((X, A)^{\wedge \sigma}; \bk)=0$ if $|J|\geq i+1$. 
If $X_1$ and $X_2$ are connected $CW$-complexes with a non-empty intersection such that $\widetilde{H}_i(X_1; \bk)=\widetilde{H}_i(X_2; \bk)=\widetilde{H}_i(X_1\cap X_2; \bk)=0$ for $i\leq l$, then $\widetilde{H}_i(X_1\cup X_2; \bk)=0$ for $i\leq l$.
As $(X, A)^{\K_{m, J}}$ is a union of $(X, A)^{\sigma}$ over all $\sigma\in \K_{m, J}$, inductively $\widetilde{H}_i((X, A)^{\wedge \K_J}; \bk)$ is trivial if $|J|\geq i+1$.

Since $\{\K_m, i_m\}$ is completely surjective, if $|J|\leq i$ there exists an integer $N\geq 1$ such that if $m\geq N$, we have $E_{m+1, i}=E_{m, i}=\ldots =E_{N, i}$, where $E_{m, i}=\{J\in E_{m}\mid |J|\leq i\}$.
Therefore the summands in~(\ref{rrpp}) do not depend on $m$ for $m\geq N$. 
On the other hand, for each $J\in E_{*}$ there exists an integer $b(J)$ such that for $m\geq b(J)$, either $\Sigma_m$ acts on $J$ trivially or the stabiliser $\stab(J, m)=\stab(J, b(J)) \times \Sigma_{m-b(J)}$, where $\Sigma_{m-|J|}$ acts on $J$ trivially. In the first case, if $\Sigma_m$ acts on $J$ trivially for $m\geq b(J)$, then for any $k\leq b(J)$, $\Sigma_k$ acts on $J$ trivially because $\Sigma_k$ acts on $J$ as a subgroup of $\Sigma_{b(J)}$. As the vertex support set $J$ of $K_{m, J}$ is fixed, the space $(X, A)^{\wedge K_{m, J}}$ will stay the same when $m$ increases. Thus, $\widetilde{H}_i((X, A)^{\wedge \K_{m, J}}; \bk)$ is a fixed finite-dimensional trivial $\Sigma_m$-representation even though $m$ varies. It follows that $\{\widetilde{H}_i((X, A)^{\wedge \K_{m, J}}; \bk)\}$ is uniformly representation stable.

If $\stab(J, m)=\stab(J, b(J)) \times \Sigma_{m-b(J)}$, then $\widetilde{H}_i((X, A)^{\wedge\K_{m, J}}; \bk)$ is a $\stab(J, m)$-representation with a trivial $\Sigma_{m-b(J)}$-action.
By~\cite{He}, we have that $\mathrm{Ind}_{\stab(J, m)}^{\Sigma_m} \widetilde{H}_i((X, A)^{\wedge \K_{m, J}}; \bk)$ is uniformly representation stable. 

Therefore, the sequence of $\Sigma_m$-representations $\{\widetilde{H}_i((X, A)^{\K_m}; \bk), i_{m_*}\}$ is uniformly representation stable as the summands do not depend on $m$ eventually.
\end{proof}

\begin{bigremark}In general, we require the simplicial maps $i_m$ in the consistent sequence of finite simplicial complexes to be inclusions so that they induce maps of polyhedral products. However, in the case when $(X,A)$ is a pair of topological monoids, as it is for moment-angle complexes when $(X, A)=(D^2, S^1)$, any $\Sigma_m$-simplicial map, not necessary a simplicial inclusion, can be chosen for $i_m$. A simplicial map $f\colon \K\longrightarrow L$ induces a continuous map $(X, A)^{\K} \longrightarrow (X, A)^{L}$ defined by $(x_1, \ldots, x_p)=(y_1, \ldots, y_q)$, where $y_j=\underset{i\in f^{-1}(j)}\prod x_i$. Here $p$ and $q$ are the number of vertices of $\K$ and $L$, respectively.
\end{bigremark}


We have proved that the sequences in Constructions~\ref{fi2} and~\ref{fi3} are completely surjective and stabiliser consistent. 
Applying Theorem~\ref{thmrrpp}, we conclude the following statement.
\begin{corollary}\label{rsconstructions}
Let $\mathcal K$ be one of the consistent sequences in Constructions~\ref{fi2} and~\ref{fi3}
and $X$ be a connected, based $CW$-complexes of finite type with a based subcomplex $A$. 
Then the consistent sequence of $\Sigma_m$-representations $\{\widetilde{H}_i((X, A)^{\K_m}; \bk), i_{m_*}\}$ for $\mathrm{char}\bk=0$  is uniformly representation stable.~\qed
\end{corollary}
Note that since the sequence in Construction~\ref{fi3} provides a consistent sequence of finite simplicial complexes, given by taking the boundary of dual of simple polytopes, the corresponding moment-angle complexes are a sequence of manifolds.

\begin{proposition}\label{rsmanifolds}
Let $\mathcal K$ be  the consistent sequence in Construction~\ref{fi3}.  Then for the moment-angle manifolds $\mathcal Z_{\mathcal K}$, the consistent sequence of $\Sigma_m$-representations $\{H_*(\mathcal Z_{\K_m}; \bk), i_{m_*}\}$ for $\mathrm{char}\bk=0$  is uniformly representation stable.~\qed
\end{proposition}
Moreover, due to~\cite{BM, CFW}, the manifold $\mathcal Z_{K_m}$ is diffeomorphic to $\partial ((\underset{m}\prod S^3-D^{3m})\times D^2)\# \underset{j=1}{\overset{m}\#}\binom{m}{j}(S^{j+2}\times S^{3m-j-1})$. Therefore, $H_3(\mathcal Z_{K_m}; \bk)$ has Betti number $m$ which means that the sequence of moment-angle manifolds $\mathcal Z_{K_m}$ with the maps $\mathcal Z_{K_m}\longrightarrow \mathcal Z_{K_{m+1}}$ induced by simplicial maps $K_m\longrightarrow K_{m+1}$ is not homology stable.

Let $K_m=\Delta_m^k$. Since every $K_m$ is a full subcomplex of $K_{m+1}$, the moment-angle complex $\mathcal{Z}_{K_m}$ retracts off $\mathcal{Z}_{K_{m+1}}$, and the retraction map $p_m\colon \mathcal{Z}_{K_{m+1}} \longrightarrow\mathcal{Z}_{K_m}$ is $\Sigma_m$-equivariant. The uniform stability of $\Sigma_m$-representations $\{H^i(\mathcal Z_{\K_m}; \bk), p_{m}^i\}$ follows immediately.

\begin{proposition}\label{stabma}
For $i\geq 2k+3$, the sequence $\{H^i(\mathcal Z_{\K_m}; \bk), p_{m}^i\}$  of $\Sigma_m$-representations  is uniformly representation stable.
\end{proposition}
\begin{proof}

By Proposition~\ref{grepi}, we have 
\[
H^i(\mathcal Z_{\K_m}; \bk) \cong \underset{J\in E_m}{\bigoplus} \mathrm{Ind}_{\Sigma_{|J|}\times \Sigma_{m-|J|}}^{\Sigma_m} \widetilde {H}^{i-|J|-1}(\K_{J, m}; \bk)
\]
where $E_m=\{\{1\}, \{1, 2\}, \ldots, \{1, 2, \ldots, m\}\}$ and $\K_{J, m}=J \cap \Delta_m^k$. Thus $\K_J$ is a $(|J|-1)$-face of $\K$, if $|J|\leq k+1$ and is the $k$-skeleton of $(|J|-1)$-simplex with $J$ as its vertex set if $|J|\geq k+2$. The latter one allows an $\Sigma_{|J|}$-action. 
Therefore, if $|J| \leq k+1$, then $\widetilde{H}^*(\K_{J, m};\bk)=0$. If $k+2\leq|J|\leq m$, $\widetilde{H}^p(\K_{J, m}; \bk)=\bk$ if $p=k$, and is $0$, otherwise.

The nontrivial cohomology group of $\K_{J, m}$ implies that $i-|J|-1=k$ and $k+2\leq|J|\leq m$. Thus if $2k+3\leq i\leq m+k+1$, we have a $\Sigma_m$-representation isomorphism 
that \[
H^i(\mathcal Z_{\K_m}; \bk)\cong \mathrm{Ind}_{\Sigma_{|J|}\times \Sigma_{m-|J|}}^{\Sigma_m} \widetilde {H}^{k}(\K_{J, m}; \bk), ~\mathrm{with}~|J|=i-k-1.
\]
Hemmer~\cite{He} implies the uniform representation stability of the sequence of $\Sigma_m$-representations $\{H^i(\mathcal{Z}_{K_m}; \bk), p_{m}^i\}$. 
\end{proof}
\begin{example}
When $K_m$ consists of $m$ disjoint points and for $i\geq 3$, as a $\Sigma_m$-representation $H^i(\mathcal{Z}_{m}; \bk)$ can be written explicitly as
\[
H^i(\mathcal{Z}_{m}; \bk)=\mathrm{Ind}_{\Sigma_{i-1}\times \Sigma_{m-i+1}}^{\Sigma_m}V_{(i-2, 1)}\boxtimes \bk
\]
where $V_{(i-2, 1)}$ is the standard representation of $\Sigma_{i-1}$.

In particular,
\[
\begin{split}
H^3(\mathcal{Z}_m; \bk)=&V_{(m-1, 1)}\oplus V_{(m-2, 1, 1)}~\text{ for}~ m\geq 3;\\
H^4(\mathcal{Z}_m; \bk)=&V_{(m-1, 1)} \oplus V_{(m-2, 2)} \oplus V_{(m-2, 1, 1)} \oplus V_{(m-3, 2, 1)} ~\text{ for}~ m\geq 5; \\
H^5(\mathcal{Z}_m; \bk)=&V_{(m-1, 1)} \oplus V_{(m-2, 2)} \oplus V_{(m-2, 1, 1)} \oplus V_{(m-3, 3)}\oplus V_{(m-3,2,1)} \oplus V_{(m-4,3,1)} ~\text{ for}~ m\geq 7; \\
H^6(\mathcal{Z}_m; \bk)=&V_{(m-1, 1)} \oplus V_{(m-2, 2)} \oplus V_{(m-2, 1, 1)} \oplus V_{(m-3, 3)}\oplus V_{(m-3,2,1)} \oplus V_{(m-4,4)}\\
                         &  \oplus V_{(m-4,3,1)}  \oplus V_{(m-5,4,1)}  ~\text{ for}~ m\geq 9.\\
\end{split}
\]
\end{example}


\section{Applications of uniformly representation stability of polyhedral products}
We finish the paper by investigating what kind of structural properties of $H_i((X,A)^K;\Q)$ are implied by representation stability.

One of the key properties of a sequence of uniformly stable $\Sigma_m$-representations over $\Q$ is that their characters are eventually polynomials~\cite[Definition 1.4]{CEF}. 

Denote by $\lambda \vdash m$ a partition $\lambda=(\lambda_1, \ldots, \lambda_l)$ with $\lambda_1\geq \ldots \lambda_l >0$ and $\lambda_1 +\ldots +\lambda_l=m$. Let $|\lambda|$ be the sum $\lambda_1 +\ldots +\lambda_l$.
Given any partition $\lambda$, for any $m\geq |\lambda|+\lambda_1$, denote by $\lambda[m]=(m-|\lambda|, \lambda_1, \ldots, \lambda_l)$ (see ~\cite[Definition 2.2.5]{CEF}).
Denote by $V(\lambda)_m$ the irreducible representation corresponding to partition $\lambda[m]$.
The \emph{weight} of a consistent sequence of $\Sigma_m$-representations $\{V_m, \psi_m\}$ is the maximum of $|\lambda|$ over all irreducible constituents $V(\lambda)_m$ that appears in $V_m$. 
\begin{example} \label{weight}
For any partition $\mu \vdash n$ and $m\geq n$, applying Proposition 3.2.4 in~\cite{CEF}, the consistent sequence $\{\mathrm{Ind}_{\Sigma_{n}\times \Sigma_{m-n}}^{\Sigma_m}V_{\mu} \boxtimes \bk\}$ has weight $n$, where $V_{\mu}$ is the irreducible representation corresponding to $\mu$.
\end{example}
Hemmer~\cite{He} constructed a sequence of $\Sigma_m$-representations that is uniformly representation stable. Next we calculate the weight of this sequence applying the result from Example~\ref{weight}.

\begin{lemma}\label{hemmer}
Fix an integer $n\geq 0$. Let $H$ be a subgroup of $\Sigma_n$ and $V$ is a $\Sigma_n$-representation over a field $\bk$ of characteristic 0. For $m\geq n$, the consistent sequence $\{\mathrm{Ind}_{H\times \Sigma_{m-n}}^{\Sigma_m}V\boxtimes \bk\}$ has weight $n$.
\end{lemma}
\begin{proof}
Observe that \[
\mathrm{Ind}_{H\times \Sigma_{m-n}}^{\Sigma_m}V\boxtimes \bk=\mathrm{Ind}_{\Sigma_n \times \Sigma_{m-n}}^{\Sigma_m}(\mathrm{Ind}_{H\times \Sigma_{m-n}}^{\Sigma_n \times \Sigma_{m-n}}(V \boxtimes \bk))=\mathrm{Ind}_{\Sigma_n \times \Sigma_{m-n}}^{\Sigma_m}(\mathrm{Ind}_{H}^{\Sigma_n}V)\boxtimes \bk.
\]
As $\Sigma_n$-representations, $\mathrm{Ind}_{H}^{\Sigma_n}V$ is decomposed as $\underset{\mu \vdash n} \bigoplus V_{\mu}^{\oplus c_{\mu}}$, where $c_{\mu}$ are multiplicities.
By Example~\ref{weight}, $\{\mathrm{Ind}_{\Sigma_{n}\times \Sigma_{m-n}}^{\Sigma_m}V_{\mu} \boxtimes \bk\}$ has weight $n$. Then $\{\mathrm{Ind}_{H\times \Sigma_{m-n}}^{\Sigma_m}V\boxtimes \bk\}$ has weight $n$, as each $\mathrm{Ind}_{H\times \Sigma_{m-n}}^{\Sigma_m}V\boxtimes \bk$ is decomposed into a finite direct sum as $\Sigma_m$-representations
\[
\mathrm{Ind}_{H\times \Sigma_{m-n}}^{\Sigma_m}V\boxtimes \bk \cong \underset{\mu \vdash n}\bigoplus c_{\mu}\mathrm{Ind}_{\Sigma_n \times \Sigma_{m-n}}^{\Sigma_m}V_{\mu}\boxtimes \bk.
\]
\end{proof}

Given a uniformly representation stable sequence $\{V_m, \psi_m\}$, the uniform multiplicity stability implies that there exists an integer $M\geq 0$ such that $V_m$ is decomposed into $\underset{\lambda}\bigoplus c_{\lambda}V(\lambda)_m$ for $m\geq M$. A classical result~(\cite[Example I.7.14]{Mac}) in representation theory states that the character of $V(\lambda)_m$ is polynomial if $m\geq |\lambda|+\lambda_1$. 
Explicitly, let $a_1, a_2, \ldots$ be class functions $a_j\colon\Sigma_i\longrightarrow\mathbb{N}$ for any $i\geq 0$ such that  $a_j(g)$ is the number of $j$-cycles in  the cycle decomposition of $g$. Then,
for each partition $\lambda$ there exists a polynomial $P_{\lambda}\in \Q[a_1, a_2, \ldots]$, called the character polynomial corresponding to the partition $\lambda$, such that $P_{\lambda}$ has degree 
$|\lambda|$ and the character $\chi_{V(\lambda)_m}(g)=P_{\lambda}(g)$ for all $m\geq |\lambda|+\lambda_1$ and $g\in \Sigma_m$.

We finish our paper by looking at the growth of Betti numbers of polyhedral products.

\begin{theorem}\label{Betti}
Let $\{K, i_m\}$ and $(X, A)$ be as in Theorem $\ref{thmrrpp}$. Then for each $i\geq 0$, the consistent sequence $\{\widetilde{H}_i((X, A)^{\K_m}; \Q), i_{m_*}\}$ has a finite weight. Moreover,
the growth of Betti numbers of $\{\widetilde{H}_i((X, A)^{\K_m}; \Q), i_{m_*}\}$ is eventually polynomial with respect to $m$.
\end{theorem}

\begin{proof}
 By Theorem~\ref{thmrrpp}, $\{\widetilde{H}_i((X, A)^{\K_m}; \Q), i_{m_*}\}$ is uniformly representation stable.
Thus the uniformly multiplicity stability implies that there exists an integer $N>0$, not depending on $\lambda$, such that for all $m\geq N$, there are constant integers $c_{\lambda}$ such that 
\[
\widetilde{H}_i((X, A)^{\K_m}; \Q)\cong \underset{\lambda}\bigoplus c_{\lambda}V(\lambda)_m
\]
and are uniquely given by multiplicities defined in the irreducible components of $\widetilde{H}_i((X, A)^{K_N}; \Q)$.
Therefore, the weight $\omega_i$ of sequence $\{\widetilde{H}_i((X, A)^{\K_m}; \Q), i_{m_*}\}$ is the maximum $|\lambda|$ that forms a irreducible component of $\widetilde{H}_i((X, A)^{K_N}; \Q)$. Since $\widetilde{H}_i((X, A)^{K_N}; \Q)$ has finite dimension over $\Q$, 
$\omega_i$ is finite.

In particular, if $m\geq 2 \omega_i$, then for all $\lambda$ appearing in the above equation, $m\geq |\lambda|+\lambda_1$. Then there exists a polynomial character of $\{\widetilde{H}_i((X, A)^{\K_m}; \Q)$ given by $\sum_{\lambda} P_{\lambda}$.
Take $g$ to be the identity of symmetric groups. This gives that the growth of Betti numbers of $\{\widetilde{H}_i((X, A)^{\K_m}; \Q), i_{m_*}\}$ is eventually polynomial with respect to $m$.
\end{proof}

{\it Acknowledgements.} The first author shows her gratitude to Benson Farb who introduced her to the subject during YTM in Stockholm in 2017. The authors also would like to thank the referee for pointing out the misleading points in the previous version and helpful comments to improve this paper.

\bibliographystyle{amsplain}
\bibliography{nummeth}

\end{document}